\documentclass[a4paper,12pt]{amsart}
\usepackage{amsmath}
\usepackage{amssymb}
\usepackage{fancybox}
\usepackage[dvips]{graphicx}
\usepackage{array}
\usepackage{url}

\def\thline{\noalign{\hrule height 1pt}}
\newtheorem{thm}{Theorem}[section]
\newtheorem{definition}[thm]{Definition}
\newtheorem{cor}[thm]{Corollary}
\newtheorem{lemma}[thm]{Lemma}
\newtheorem{prop}[thm]{Proposition}

\newtheorem*{rem}{Remark}

\title[Train track complex]
{Train track complex of once-punctured torus and 4-punctured sphere}
\author{Keita Ibaraki}
\address{Department of Mathematical and Computing Sciences,
Tokyo Institute of Technology, Tokyo 152-8552, Japan}
\email{ibaraki4@is.titech.ac.jp}

\date{\today}

\begin{document}

\begin{abstract}
Consider a compact oriented surface $S$ of genus $g \geq 0$ and $m \geq 0$ punctured.
The train track complex of $S$ which is defined by Hamenst\"adt
is a 1-complex whose vertices are isotopy classes of complete train tracks on $S$.
Hamenst\"adt shows that if $3g-3+m \geq 2$, 
the mapping class group acts properly discontinuously and cocompactly on the train track complex.
We will prove corresponding results for the excluded case, 
namely when $S$ is a once-punctured torus or a 4-punctured sphere. 
To work this out,
we redefinition of two complexes for these surfaces.
\end{abstract}

\maketitle

\tableofcontents

\section{Introduction}
Consider a compact oriented surface $S$ of genus $g \geq 0$ from which $m \geq 0$ points, 
so-called punctures, have been deleted. 
The {\itshape mapping class group} $\mathcal{M}(S)$ of $S$ is, by definition, the space of
isotopy classes of orientation preserving homeomorphisms of $S$.

There are natural metric graphs on which the mapping class group $\mathcal{M}(S)$ acts by isometries.
Among them, we will concern with the {\itshape curve complex} $\mathcal{C}(S)$ 
(or, rather, its one-skeleton) 
and {\itshape train track complex} $\mathcal{TT}(S)$.

In \cite{Ha:81}, Harvey defined the {\itshape curve complex} $\mathcal{C}(S)$ for $S$. 
The vertex of this complex is a free homotopy class of an essential simple closed curve on $S$, 
i.e. a simple closed curve which is neither contractible nor homotopic into a puncture. 
A $k$-simplex of $\mathcal{C}(S)$ is spun by a collection of $k+1$ vertices 
which are realized by mutually disjoint simple closed curves.

The {\itshape train track} which is an embedded 1-complex was invented by Thurston \cite{Th:80}
and provides a powerful tool for the investigation of surfaces and hyperbolic $3$-manifolds. 
A detailed account on train tracks can be found in the book \cite{PH:92} of Penner with Harer. 

Hamenst{\"a}dt defined in \cite{Ham:08} that 
a train track is called {\itshape complete} if it is a bireccurent 
and each of complementary regions is either a trigon or a once-punctured monogon. 
Hamenst{\"a}dt defined the {\itshape train track complex} $\mathcal{TT}(S)$ for $S$. 
Vertices of the train track complex are isotopy classes of complete train tracks on $S$.

Suppose $3g-3+m \geq 2$, i.e. S is a hyperbolic surfaces but neither a once-punctured torus nor 4-punctured sphere. 
Both the curve complex and the train track complex can be endowed with a path-metric
by declaring all edge lengths to be equal to 1.
In these cases,
there are a map $\Phi : \mathcal{TT}(S) \rightarrow \mathcal{C}(S)$ and a number $C > 0$ such that 
$d(\Phi(\tau), \Phi(\tau')) \leq Cd(\tau, \tau')$ for all complete train track $\tau, \tau'$ on $S$.

If $S$ is a once-punctured torus or a 4-punctured sphere,
then any essential simple closed curve must intersect, 
so that $\mathcal{C}(S)$ has no edge by definition.
However, Minsky \cite{Min:96} adopt a small adjustment in the definition for curve complex 
in these two particular case so that it becomes a sensible and familiar 1-complex:
two vertices are connected by an edge when the curves they represent have minimal intersection 
(1 in the case of a once-punctured torus, and 2 in the case of a 4-punctured sphere).
It turns out that in both cases, the complex is the {\itshape Farey graph} $\mathcal{F}$.

In addition, if $S$ is a once-punctured torus,
there is no complete train track under Hamenst{\"a}dt's definition
and $\mathcal{TT}(S)$ is homeomorphic to empty set.
We thus adopt here Penner's definition \cite{PH:92}, 
i.e. the train track $\tau$ is complete iff $\tau$ is bireccurent and 
is not a proper subtrack of any birecurrent train track.
These two definitions are equivalent except for the once-punctured torus.

Our main theorem is:

\begin{thm}\label{thm:main}
Suppose $S$ is a once-punctured torus or a 4-punctured sphere.
Then the train track complex $\mathcal{TT}(S)$ of $S$
is quasi-isometric to the dual graph of the Farey graph $\mathcal{F}$ (see Figure\ref{fig:dual}).
\end{thm}

\begin{figure}[htbp]
  \begin{center}
    \includegraphics[keepaspectratio=true,width=75mm]{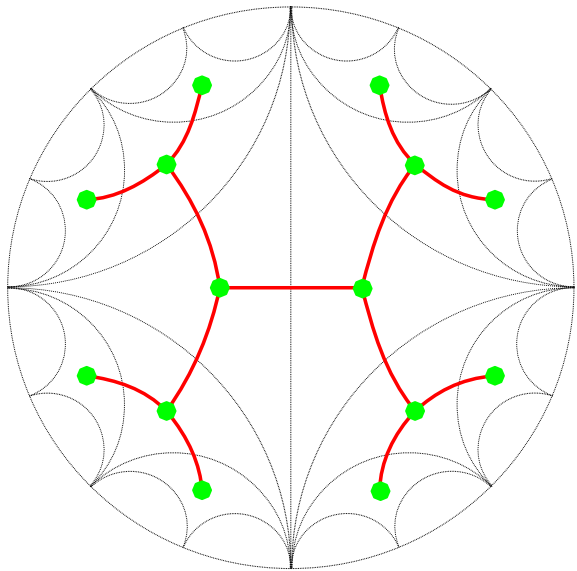}
  \end{center}
  \caption{dual of the Farey graph}
  \label{fig:dual}
\end{figure}

More precisely, if $S$ is a once-punctured torus, 
we show:
\begin{thm}\label{prop:torus}
The train track complex of a once-punctured torus is isomorphic to the Caley graph of
$PSL(2,\mathbb{Z}) = \langle r,l \mid (lr^{-1}l)^2=1, (lr^{-1})^3=1 \rangle$.
\end{thm}

In \cite{Ham:08}, Hamenst{\"a}dt also shows that 
if $3g-3+m \geq 2$ the mapping class group 
$\mathcal{M}(S)$ acts p.d.c., i.e. properly discontinuously and cocompactly, 
on the train track complex $\mathcal{TT}(S)$
and $\mathcal{M}(S)$ is quasi-isometric to $\mathcal{TT}(S)$.
We can prove that the same is true for the once-punctured torus 
and 4-punctured sphere. 

\begin{cor}\label{cor:pdc}
Suppose $S$ is the once-punctured torus or 4-punctured sphere. 
Then $\mathcal{M}(S)$ acts p.d.c. on $\mathcal{TT}(S)$. 
\end{cor}

\begin{cor}\label{cor:qi}
Suppose $S$ is the once-punctured torus or 4-punctured sphere. 
Then  $\mathcal{M}(S)$ is quasi-isometric to $\mathcal{TT}(S)$. 
\end{cor}

In Section \ref{sec:qi}, we give a brief review of quasi-isometries. 
In Section \ref{sec:ttc}, we describe train tracks and define the train track complex. 
In Section \ref{sec:farey}, we describe how to build a Farey graph which is used for 
curve complex of a once-punctured torus or a 4-punctured sphere. 
In Section \ref{sec:torus} and \ref{sec:4sphere}, we prove the Theorem\ref{thm:main}. 
Finally, we describe the action of mapping class groups 
on train track complexes in Section \ref{sec:act}. 

\section{Quasi-Isometry}\label{sec:qi}
A quasi-isometry is one of the fundamental notion in geometric group theory.
For details, see \cite{Bow:06}.

Let $(X,d)$ be a proper geodesic space, i.e. a complete and locally compact geodesic space.
Given $x \in X$ and $r \geq 0$, write $N(x, r) = \{ y \in X \mid d(x,y) \leq r \}$
for the closed $r$-neighborhood of $x$ in $X$.
If $A \subseteq X$, write $N(A, r) = \bigcup_{x \in A}N(x, r)$.
We say that $A$ is {\itshape cobounded} if $X = N(A,r)$ for some $r \geq 0$.

Suppose that a group $\Gamma$ acts on X by isometry.
Given $x \in X$, we write $\Gamma x = \{ gx \mid g \in \Gamma \}$ 
for the {\itshape orbit} of $x$ under $\Gamma$, 
and $stab(x)=\{ g \in \Gamma \mid gx = x \}$ for its {\itshape stabilizer}.

We say that the action of $\Gamma$ on $X$ is {\itshape properly discontinuous}
if for all $r \geq 0$ and all $x \in X$,
the set $\{ g \in \Gamma \mid d(x, gx) \leq r \}$ is finite.
A properly discontinuous action is called {\itshape cocompact}
if $X/\Gamma$ is compact.
We will frequently abbreviate ``properly discontinuous and cocompact'' by p.d.c.

\begin{prop}[\cite{Bow:06}]\label{ex:bow}
The followings are equivalent:
\begin{enumerate}
\item The action is cocompact,
\item Some orbit is cobounded, and
\item Every orbit is cobounded.
\end{enumerate}
\end{prop}
\begin{proof}
Write $N'(\Gamma x, r)$ for a closed $r$-neighborhood of $\Gamma x$ in $X/\Gamma$.
Let $\pi : X \rightarrow X / \Gamma$ be a quotient map,
Then for any $x \in X$ and any $r > 0$
$\pi(N(x,r))=N'(\Gamma x, r)$ and $\pi^{-1}(N'(\Gamma x, r))= N(\Gamma x, r)$.
\begin{itemize}
\item (iii) $\Rightarrow$ (ii) is clear.
\item Suppose that some orbit is cobounded.
So, $X=N(\Gamma x, r)$ for some $x \in X$ and some $r > 0$.
Thus, $X/\Gamma=\pi(X)=\pi(N(\Gamma x, r))=N'(\Gamma x, r)=\pi(N(x,r))$.
By Proposition 3.1 of \cite{Bow:06}, $N(x,r)$ is compact
and hence $X/\Gamma=\pi(N(x,r))$ is also compact.
Now we proved (ii) $\Rightarrow$ (i).
\item Suppose the action is cocompact.
So, $X/\Gamma$ is compact and hence $X/\Gamma$ is bounded.
Thus, for any $x \in X$ there is some $r>0$, such that $X/\Gamma=N'(\Gamma x,r)$.
Since $X=\pi^{-1}(X/\Gamma)=\pi^{-1}(N'(\Gamma x,r))=N(\Gamma x,r)$,
$\Gamma x$ is cobounded and (i) $\Rightarrow$ (iii) is shown.
\end{itemize}
\end{proof}

\begin{definition}[quasi-isometry]
Let $(X,d)$ and $(Y,d')$ be metric spaces.
A map $\varphi : X \rightarrow Y$ is called a quasi-isometry 
if there are constants $k_1 > 0, k_2,k_3,k_4 \geq 0$ such that for all $x_1, x_2 \in X$,
\[ k_1 d(x_1,x_2)-k_2 \leq d'(\varphi(x_1), \varphi(x_2)) \leq k_3 d(x_1,x_2)+k_4, \]
and the image $\phi(x)$ is cobounded in $Y$.
\end{definition}

Thus, a quasi-isometry is bi-Lipshitz with bounded error 
and its image is cobounded.
We note that the quasi-isometry introduces an equivalence relation on the set of metric spaces.

Two metric spaces, $X$ and $Y$, are said to be {\itshape quasi-isometric}
if there is a quasi-isometry between them.

Let $X$ be a geodesic space and $A$ a finite generating set for a group $\Gamma$.
Suppose $\Delta(\Gamma, A)$ is the Cayley graph of $\Gamma$ with respect to $A$.
If $B$ is another generating set for $\Gamma$,
then $\Delta(\Gamma, A)$ is quasi-isometric to $\Delta(\Gamma, B)$.
Thus, we simply denote the Cayley graph of $\Gamma$ by $\Delta(\Gamma)$
without specifying a generating set.
A group $\Gamma$ acts p.d.c. on its Cayley graph $\Delta(\Gamma)$.

We define that $\Gamma$ is quasi-isometric to $X$
if $\Delta(\Gamma)$ is quasi-isometric to $X$.
Also, two groups $\Gamma$, $\Gamma'$ are quasi-isometric
if $\Delta(\Gamma)$ is quasi-isometric to $\Delta(\Gamma')$.

The proof of the following claims can be found for example in \cite{Bow:06}:
\begin{thm}[\cite{Bow:06}]\label{thm:bow}
If $\Gamma$ acts p.d.c. on a proper geodesic space $X$,
then $\Gamma$ is quasi-isometric to $X$.
\end{thm}
\begin{prop}\label{prop:bow}
Let $\Gamma$ be a finitely generated group.
Suppose that $G$ is a subgroup of $\Gamma$ of finite index.
Then $G$ is finitely generated and quasi-isometric to $\Gamma$.
\end{prop}

\section{Train track complex}\label{sec:ttc}
A {\itshape train track} on $S$ (see \cite{PH:92}) is
an embedded 1-complex $\tau \subset S$ whose vertecis are called {\itshape switches} and 
edges are called {\itshape branches}.
$\tau$ is $C^1$ away from its switches.
At any switch $v$ the incident edges are mutually tangent 
and there is an embedding $f:(0,1) \rightarrow \tau$ with $f(1/2) = v$ which is a $C^1$ map into $S$.
The valence of each switch is at least $3$, 
except possibly for one bivalent switch in a closed curve component. 
Finally, we require that every component $D$ of $S - \tau$ has negative generalized Euler characteristic
in the following sense: 
define $\chi'(D)$ to be the Euler characteristic $\chi(D)$ minus $1/2$ 
for every outward-pointing cusp (internal angle $0$).
For the train track complementary regions all cusps are outward,
so that the condition $\chi'(D) < 0$ excludes annuli, once-punctured disks with smooth boundary,
or non-punctured disks with $0, 1$ or $2$ cusps at the boundary.
We will usually consider isotopic train-tracks to be the same.

A train track is called {\itshape generic} if all switches are at most trivalent. 
A {\itshape train route} is a non-degenerate smooth path in $\tau$; 
in particular it traverses a switch only by passing from incoming to outgoing edge or vice versa. 
The train track $\tau$ is called {\itshape recurrent} if every branch is contained in a closed train route.
The train track $\tau$ is called {\itshape transversely recurrent} if every branch intersects 
transversely with a simple closed curve $c$ so that $S - \tau - c$ 
does not contain an embedded bigon, i.e. a disc with two corners. 
A train track which is both recurrent and transversely recurrent is called {\itshape birecurrent}.

A curve $c$ is {\itshape carried} by a transversely recurrent train track $\tau$
if there is a {\itshape carrying map} $F : S \rightarrow S$ of class $C^1$
which is homotopic to the identity and maps $c$ to $\tau$ in such a way 
that the restriction of its differential $dF$ to every tangent line of $c$ is non-singular.
A train track $\tau'$ is {\itshape carried} by $\tau$
if there is a carrying map $F$ and every train route on $\tau'$ is carried by $\tau$ with $F$.

A generic birecurrent train track is called {\itshape complete} if it is
not a proper subtrack of any birecurrent train track.

\begin{thm}[{\cite{PH:92}}]\label{thm:ph}
\begin{enumerate}
\item If $g > 1$ or $m > 1$, then any birecurrent train track on $S$ is a
subtrack of a complete train track, each of whose complementary region is
either a trigon or a once-punctured monogon.
\item Any birecurrent train track on a once-punctured torus is
a subtrack of a complete train track
whose unique complementary region is a once-punctured bigon.
\end{enumerate}
\end{thm}

It follows:

\begin{cor}\label{cor:switch}
Suppose $\tau$ is a complete train track on $S$ of genus $g$ with $m$ punctures.
Then the number of switches of $\tau$ depends only on the topological type of $S$.
\end{cor}

\begin{proof}
If $S$ is the once-punctured torus, then $\tau$ have 2 vertices.

In the other case,
let $n_t$ be the number of triangle component of $S - \tau$,
$n_s$ be the number of switches of $\tau$ and
$n_b$ be the number of branches of $\tau$.
Since $\tau$ is generic, $2n_b = 3n_s$. 
By Theorem \ref{thm:ph}, $n_s = 3n_t + m$.
By Euler characteristic, $n_t-n_b+n_s=2-2g-m$.
Now we get $n_s = 4(3g-3+m)$. 
\end{proof}

A half-branch $\tilde{b}$ in a generic train track $\tau$ incident on a switch $v$ is called {\itshape large}
if the switch $v$ is trivalent and if every arc $\rho : (−\varepsilon, \varepsilon) \rightarrow \tau$ 
of class $C^1$ which passes through $v$ meets the interior of $\tilde{b}$.
A branch $b$ in $\tau$ is called {\itshape large} if each of its two half-branches is large;
in this case $b$ is necessarily incident to two distinct switches.

There is a simple way to modify a complete train track $\tau$ to another complete train track.
Namely, if $e$ is a large branch of $\tau$ then we can perform a right or left {\itshape split} of $\tau$ 
at $e$ as shown in Figure \ref{fig:split}.
A complete train track $\tau$ can always be at least one of the left or right split 
at any large branch $e$ to a complete train track $\tau'$(see \cite{Ham:05}).
We note that $\tau'$ is carried by $\tau$.
\begin{figure}[htbp]
  \begin{center}
    \includegraphics[keepaspectratio=true,width=110mm]{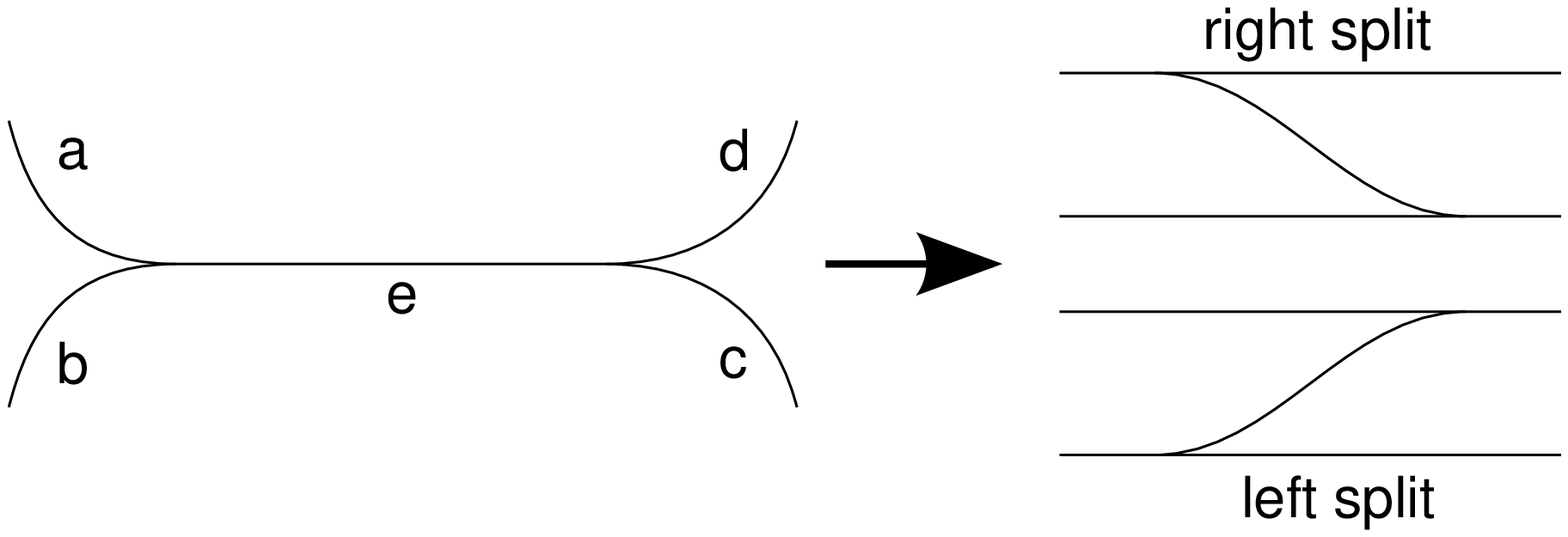}
  \end{center}
  \caption{a split}
  \label{fig:split}
\end{figure}

\begin{definition}[train track complex]
A train track complex $\mathcal{TT}(S)$ is defined as follow: 
The set of vertices of $\mathcal{TT}(S)$ consists of all isotopy classes of complete train tracks on $S$. 
Two Complete train tracks $\tau, {\tau}'$ is connected with an edge
if $\tau'$ can be obtained from $\tau$ by a single split.
\end{definition}

For each switch $v$ of $\tau$, fix a direction of the tangent line to $\tau$ at $v$.
The branch $b$ which is incident to $v$ is called {\itshape incoming}
if the direction at $v$ coincides with the direction from $b$ to $v$, and {\itshape outgoing} if not.
A {\itshape transverse measure} on $\tau$ is a non-negative function $\mu$ on the set of branches 
satisfying the switch condition : For any switch of $\tau$ the sums of $\mu$
over incoming and outgoing branches are equal.
A train track $\tau$ is recurrent if and only if
it supports a transverse measure which is positive on every branch (see \cite{PH:92}). 

For a recurrent train track $\tau$ the set $P(\tau)$ of all transverse measures on $\tau$
is a convex cone in a linear space. 
A {\itshape vertex cycle} (see \cite{MM:99}) on $\tau$ is a transverse measure $\mu$ 
which spans an extremal ray in $P(\tau)$.
Up to scaling, a vertex cycle $\mu$ is a {\itshape counting measure}
of a simple closed curve $c$ which is carried by $\tau$.
This means that for a carrying map $F:c \rightarrow \tau$ and
every open branch $b$ of $\tau$ the $\mu$-weight of $\tau$
equals the number of connected components of $F^{-1}(b)$.
We also use the notion, a vertex cycle, for the simple closed curve $c$.

\begin{prop}[\cite{Ham:05}]\label{prop:ham}
Let $\tau$ be a complete train track.
Suppose $c$ is a vertex cycle on $\tau$ with a carrying map $F$.
Then, $F(c)$ passes through every branch of $\tau$ at most twice,
and with different orientation if any.
\end{prop}

Proposition \ref{prop:ham} and Corollary \ref{cor:switch} imply that 
the number of vertex cycles on a complete train track on $S$ 
is bounded by a universal constant (see \cite{MM:99}).
Moreover, there is a number $D > 0$ with the property that
for every complete train track $\tau$ on $S$
the distance in $\mathcal{C}(S)$ between any two vertex cycles on $\tau$
is at most $D$ (see \cite{Ham:05}, \cite{MM:04}).

\section{Farey graph}\label{sec:farey}
Let S be the once-punctured torus or the 4-punctured sphere. 
The essential simple closed curves on S are well known to be in 
one-to-one correspondence with rational numbers $p/q$ with $1/0=\infty$ 
Thus the $0$-skeleton of $\mathcal{C}(S)$ is identified with $\hat{\mathbb{Q}} := \mathbb{Q} \cup \{ \infty \}$ 
in the circle $S^1=\mathbb{R} \cup \infty$. 

There are numerous ways to build a Farey graph $\mathcal{F}$, 
any of them produces an isomorphic graph. 
One can start with the rational projective line $\hat{\mathbb{Q}}$, 
identifying $0$ with $0/1$ and $\infty$ with $1/0$, 
and take this to be the vertex set of $\mathcal{F}$. 
Then, two projective rational numbers $p/q, r/s \in \hat{\mathbb{Q}}$, 
where $p$ and $q$ are coprime and $r$ and $s$ are coprime, 
are deemed to span an edge, or 1-simplex, if and only if $| ps - rq | = 1$. 
The result is a connected graph in which every edge separates. 
The graph $\mathcal{F}$ can be represented on a disc; 
see Figure \ref{fig:farey}. 
\begin{figure}[htb]
  \begin{center}
    \includegraphics[keepaspectratio=true,width=81mm]{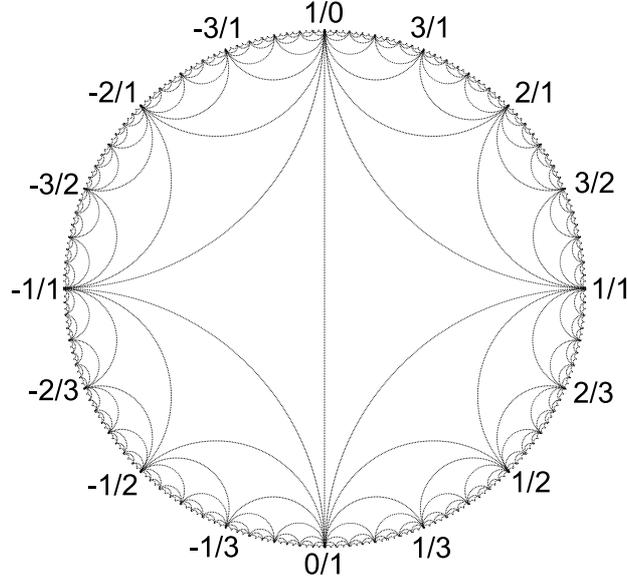}
  \end{center}
  \caption{the Farey graph}
  \label{fig:farey}
\end{figure}
We shall say a graph is a Farey graph if it is isomorphic to $\mathcal{F}$. 

Note that the curve complexes of a once-punctured torus and 4-punctured sphere 
are Farey graphs (see \cite{Min:96}\cite{Ken:06}). 

\begin{rem}
The Farey graph $\mathcal{F}$ is quasi-isometric to the dual graph.
\end{rem}

\section{Train tracks on the once-punctured torus}\label{sec:torus}
Let $S_{1,1}$ be the once-punctured torus and $\tau$ a complete train track on $S_{1,1}$. 
By Theorem \ref{thm:ph} and Corollary \ref{cor:switch},
$\tau$ has the unique complementary region that is a once-punctured bigon and 
the number of switches of $\tau$ equals 2. 
It follows that every complete train track on $S_{1,1}$ is orientation preserving $C^1$--diffeomorphic to 
the one illustrated in Figure \ref{fig:11-tt-1}. 

\begin{figure}[htbp]
  \begin{center}
    \includegraphics[keepaspectratio=true,height=40mm]{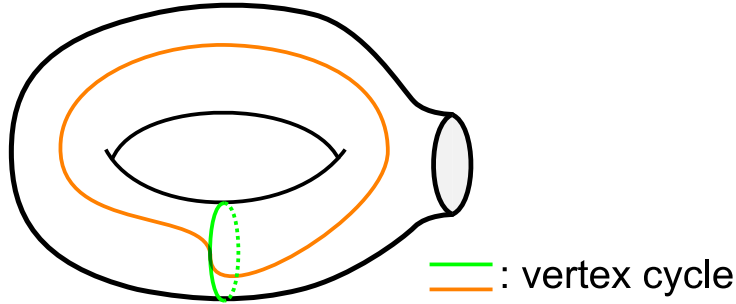}
  \end{center}
  \caption{complete train track on Once-punctured torus}
  \label{fig:11-tt-1}
\end{figure}

$\tau$ has exactly two vertex cycles $c_1, c_2$ whose
{\itshape intersection number} $i(c_1, c_2)$ equals $1$. 
Thus, $c_1$ and $c_2$ is connected by an edge in Farey graph.
Conversely,
if we fix simple closed curve $c_1, c_2$ on $S_{1,1}$ whose intersection number $i(c_1, c_2)$ equals $1$,
then there is only two complete train tracks whose vertex cycles are $c_1$ and $c_2$. 

Write $V(G)$ for the vertex set of a graph $G$ and $E(G)$ for the set of all edges in G.
We define a map $\varphi : V(\mathcal{TT}(S_{1,1})) \rightarrow E(\mathcal{F})$ as follow:
Let $\tau \in S_{1,1}$. Suppose $c_1, c_2$ are vertex cycles on $\tau$.
We define $\varphi(\tau)$ as the edge of $\mathcal{F}$ which connects $c_1$ and $c_2$. 

We construct the graph $G_1$ of $\varphi$ as follows: 
Let $V(G_1) = E(\mathcal{F})$. 
We connect vertices $e_1, e_2 \in E(\mathcal{F})$ if 
some $\tau_1 \in \varphi^{-1}(e_1)$ and some $\tau_2 \in \varphi^{-1}(e_2)$ 
are connected by an edge in $\mathcal{TT}(S_{1,1})$. 
$\varphi$ can be naturally extended to $\varphi : \mathcal{TT}(S_{1,1}) \rightarrow G_1$.

\begin{lemma}\label{lemma:QuasiG1}
$G_1$ is quasi-isometric to the dual graph of the Farey graph. 
\end{lemma}

\begin{proof}
Suppose $\tau_1$ and $\tau_2$ are different complete train tracks on $S_{1,1}$
which have common vertex cycles $c_1, c_2$.
$\tau_1$ have a unique large edge $b$ and 
there are two complete train tracks $\tau'_1, \tau''_1$ which are obtained
by a left or right split of $\tau_1$ at $b$ \ (see Figure \ref{fig:11-split}).
\begin{figure}[tb]
  \begin{center}
    \includegraphics[keepaspectratio=true,width=100mm]{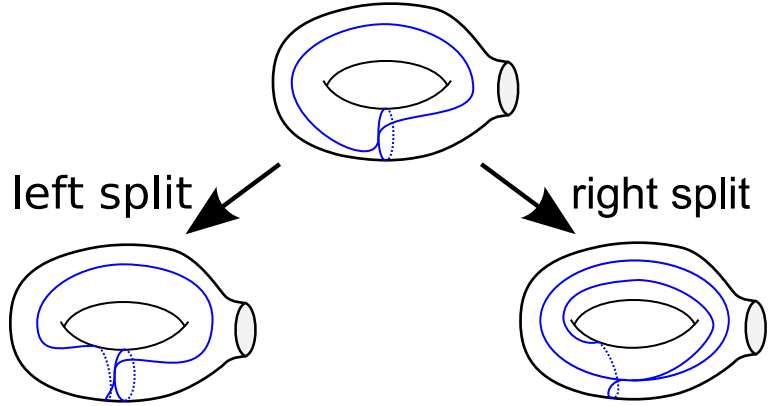}
  \end{center}
  \caption{Splits of a complete train track on $S_{1,1}$}
  \label{fig:11-split}
\end{figure}
We can see that one vertex cycle on $\tau'_1$ is the same as $c_1$ on $\tau_1$, 
and the other vertex cycle $c_3$ on $\tau'_1$ intersects $c_2$ on $\tau_1$ at one point.
Thus $\varphi(\tau'_1)$ and $\varphi(\tau_1)$ are adjacent edges in $\mathcal{F}$. 
In this case, vertex cycles on $\tau''_1$ are $c_2$ and $c_3$.
Hence $\varphi(\tau_1)$, $\varphi(\tau'_1)$ and $\varphi(\tau''_1)$ span 
a triangle on $\mathcal{F}$.
Similarly, we can get complete train tracks $\tau'_2, \tau''_2$ by a split of $\tau_2$,
 and $\varphi(\tau_2)$, $\varphi(\tau'_2)$ and $\varphi(\tau''_2)$ 
span another triangle on $\mathcal{F}$. (see Figure \ref{fig:11-split-f})

The mapping class group $\mathcal{M}(S_{1,1})$ acts isometrically on $\mathcal{TT}(S_{1,1})$ 
and $\mathcal{F}$ and acts transitively on $V(G_1)=E(\mathcal{F})$.
It follows that every edge in $G_1$ connects an adjacent edge of $\mathcal{F}$ and 
every adjacent edge of $\mathcal{F}$ is connected with a direct edge in $G_1$.
Thus $G_1$ is the line graph of the dual of $\mathcal{F}$, 
i.e. vertices of $G_1$ represent edges of the dual of $\mathcal{F}$ and 
two vertices are adjacent iff their corresponding edges share a common endpoint (see Figure \ref{fig:g1}).
It's now obvious that $G_1$ is quasi-isometric to the dual of the Farey graph.
\end{proof}

\begin{figure}[htb]
   \begin{minipage}{0.49\hsize}
      \begin{center}
        \includegraphics[keepaspectratio=true,width=65mm]{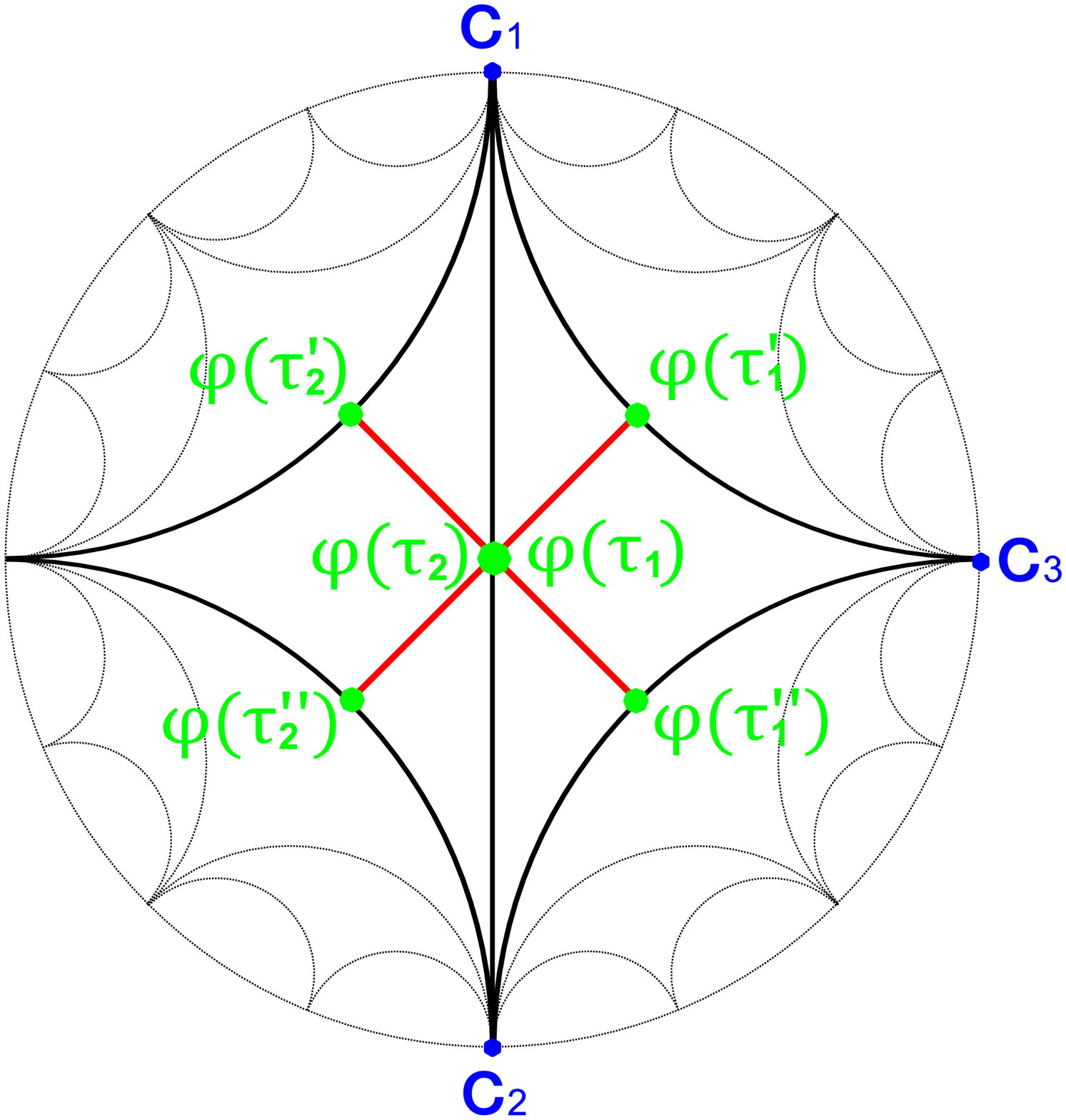}
        \caption{}
        \label{fig:11-split-f} 
     \end{center}
   \end{minipage}
   \begin{minipage}{0.49\hsize}
      \begin{center}
        \includegraphics[keepaspectratio=true,width=65mm]{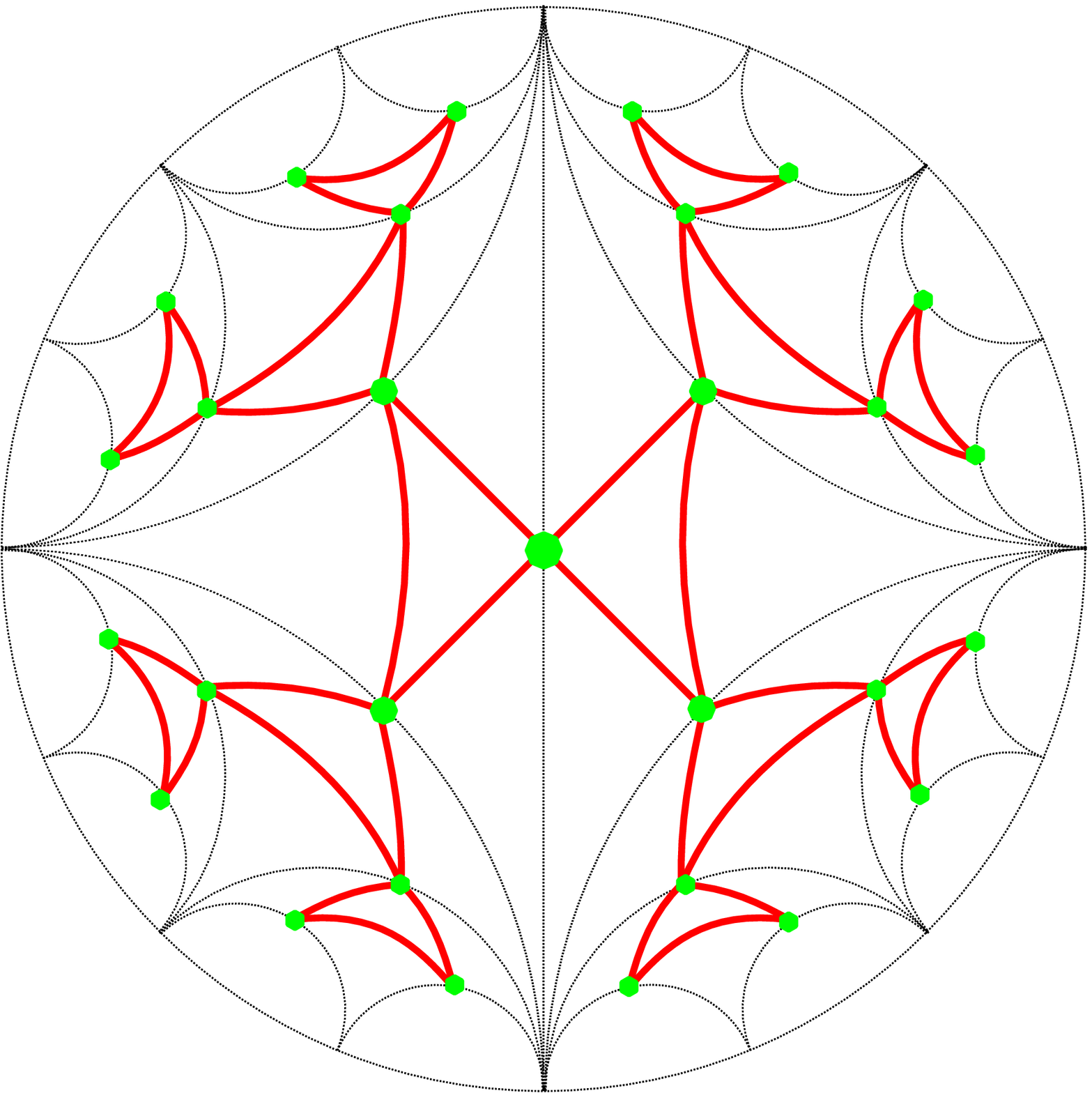}
        \caption{$G_1$}
        \label{fig:g1} 
     \end{center}
   \end{minipage}
\end{figure}

\begin{lemma}\label{lemma:connect}
$\mathcal{TT}(S_{1,1})$ is connected.
\end{lemma}
\begin{proof}
Let $e \in E(\mathcal{F})$ and $\varphi^{-1}(e):= \{\tau, \tau' \}$.
All we need is to show that $\tau$ and $\tau'$ are connected in $\mathcal{TT}(S_{1,1})$, 
since $G_1$ is connected.
$\tau$ can be a right(left) split to a complete train track $\tau_1$. 
Then, there is a complete train track $\tau_2$ which 
can be a right(left) split to $\tau'$ and 
can be a left(right) split to $\tau_1$ (see Figure \ref{fig:11-split2}).
Thus $\tau$ and $\tau'$ are connected and $d(\tau, \tau') = 3$.
\begin{figure}[htb]
  \begin{center}
     \includegraphics[keepaspectratio=true,width=90mm]{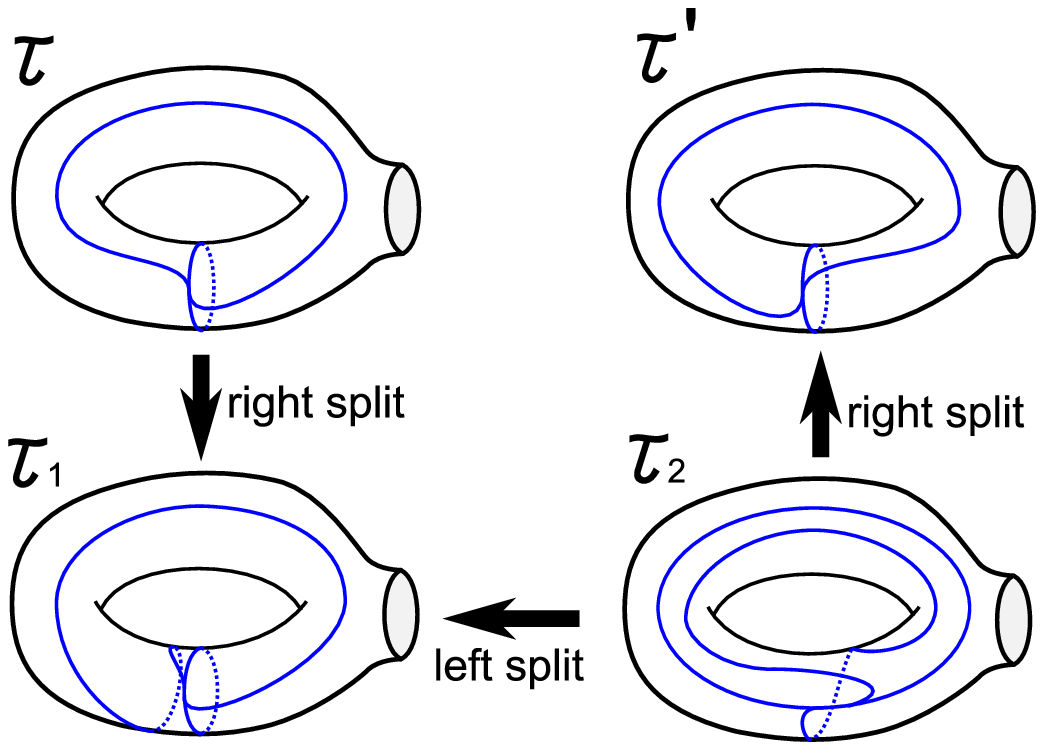}
  \end{center}
  \caption{}
  \label{fig:11-split2}
\end{figure}
\end{proof}

\begin{lemma}\label{lemma:QuasiPhi}
$\varphi$ is a quasi-isometory.
\end{lemma}
\begin{proof}
Let $\tau, \tau' \in V(\mathcal{TT}(S_{1,1}))$. 
Suppose $\alpha$ is geodesic on $G_1$ from $\varphi(\tau)$ to $\varphi(\tau')$.
$\tau, \tau' \in \varphi^{-1}(\alpha)$ and
$diam(\varphi^{-1}(\alpha)) \leq 4d(\varphi(\tau), \varphi(\tau'))+3$ 
since $diam(\varphi^{-1}(e))=3$ for any $e \in E(\mathcal{F})$.
Thus $d(\tau, \tau') \leq 4d(\varphi(\tau), \varphi(\tau'))+3$.

It follows that $\varphi$ is a quasi-isometry.
\end{proof}

\begin{proof}[Proof of Theorem \ref{thm:main}(a once-puncture torus case)]
By Lemma \ref{lemma:QuasiG1} and \ref{lemma:QuasiPhi},
$\mathcal{TT}(S_{1,1})$ is quasi-isometric to the dual graph of the Farey graph.
\end{proof}

$\mathcal{TT}(S_{1,1})$ is obtained by extending one vertex of $G_1$ to two vertices.
When we think the action of the mapping class group,
we see that $\mathcal{TT}(S_{1,1})$ is isomorphic to the graph as in Figure \ref{fig:g2}.
We can notice that this graph is isomorphic to the Caley graph of
$PSL(2,\mathbb{Z}) = \langle r,l \mid (lr^{-1}l)^2=1, (lr^{-1})^3=1 \rangle$,
and thus Theorem \ref{prop:torus} is proved.

\begin{figure}[htb]
  \begin{center}
    \includegraphics[keepaspectratio=true,width=80mm]{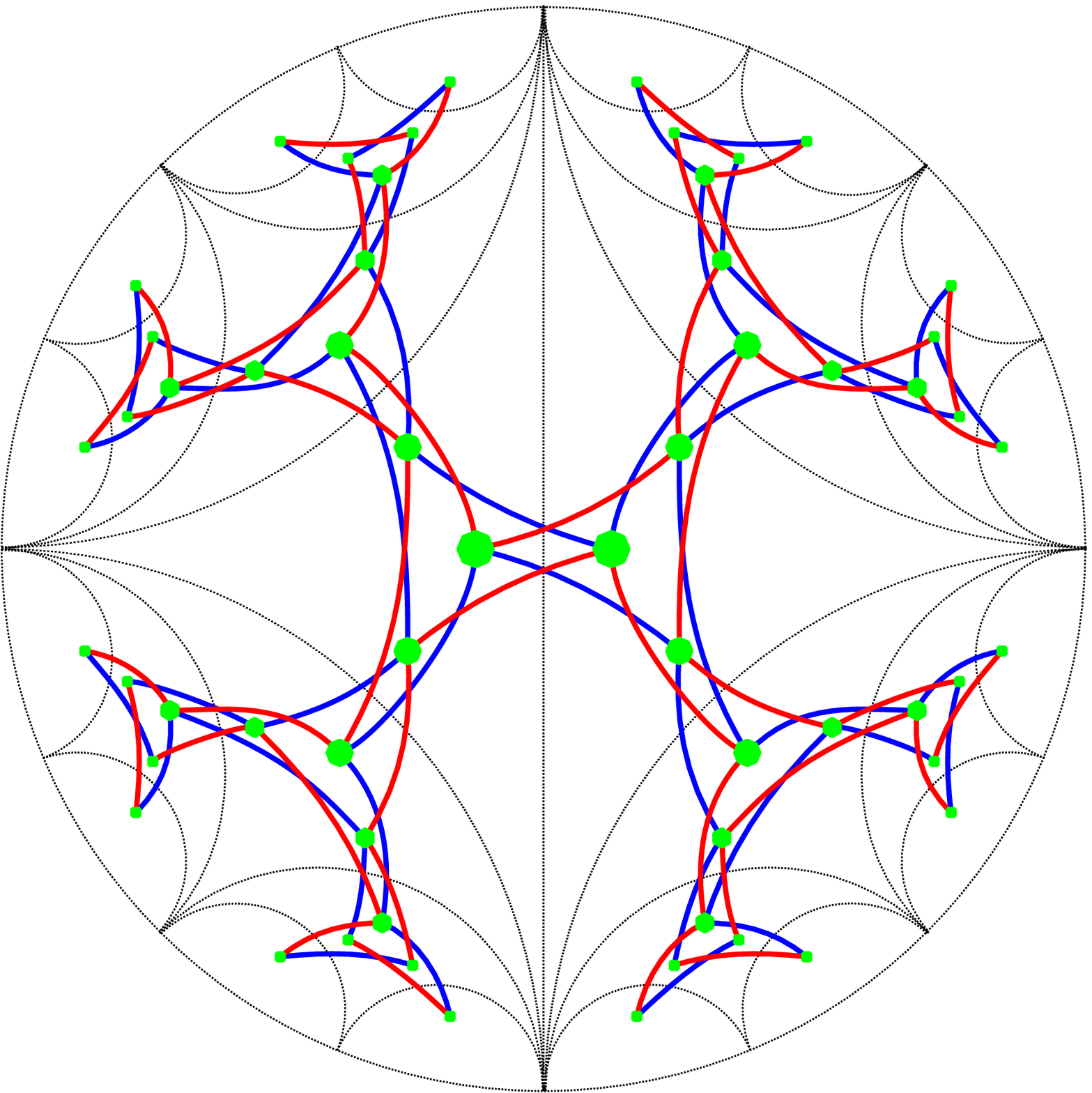}
  \caption{}
  \label{fig:g2} 
  \end{center}
\end{figure}

\section{Train tracks on the 4-punctured sphere}\label{sec:4sphere}
Let $S_{0,4}$ be the 4-punctured sphere.
A train track complex of $S_{0,4}$ is similar to that of the once-punctured torus but more complicated. 

Orientation preserving $C^1$--diffeomorphism classes of a complete train tracks $\tau$ 
depends on combination of switches and branches. 
By Proposition \ref{cor:switch}, number of switches and branches of $\tau$ are constants. 
Thus, number of orientation preserving $C^1$--diffeomorphism classes of the complete train tracks is finite. 
In fact, complete train tracks on $S_{0,4}$ are classified into 13 classes,
illustrated in (1) to (8) of Figure \ref{fig:04-tts} and their mirror images, 
though (1), (4) and (8) can move these mirrors by orientation preserving $C^1$-diffeomorphism.
 
\begin{figure}[htbp]
  \begin{center}
    \includegraphics[scale=0.92]{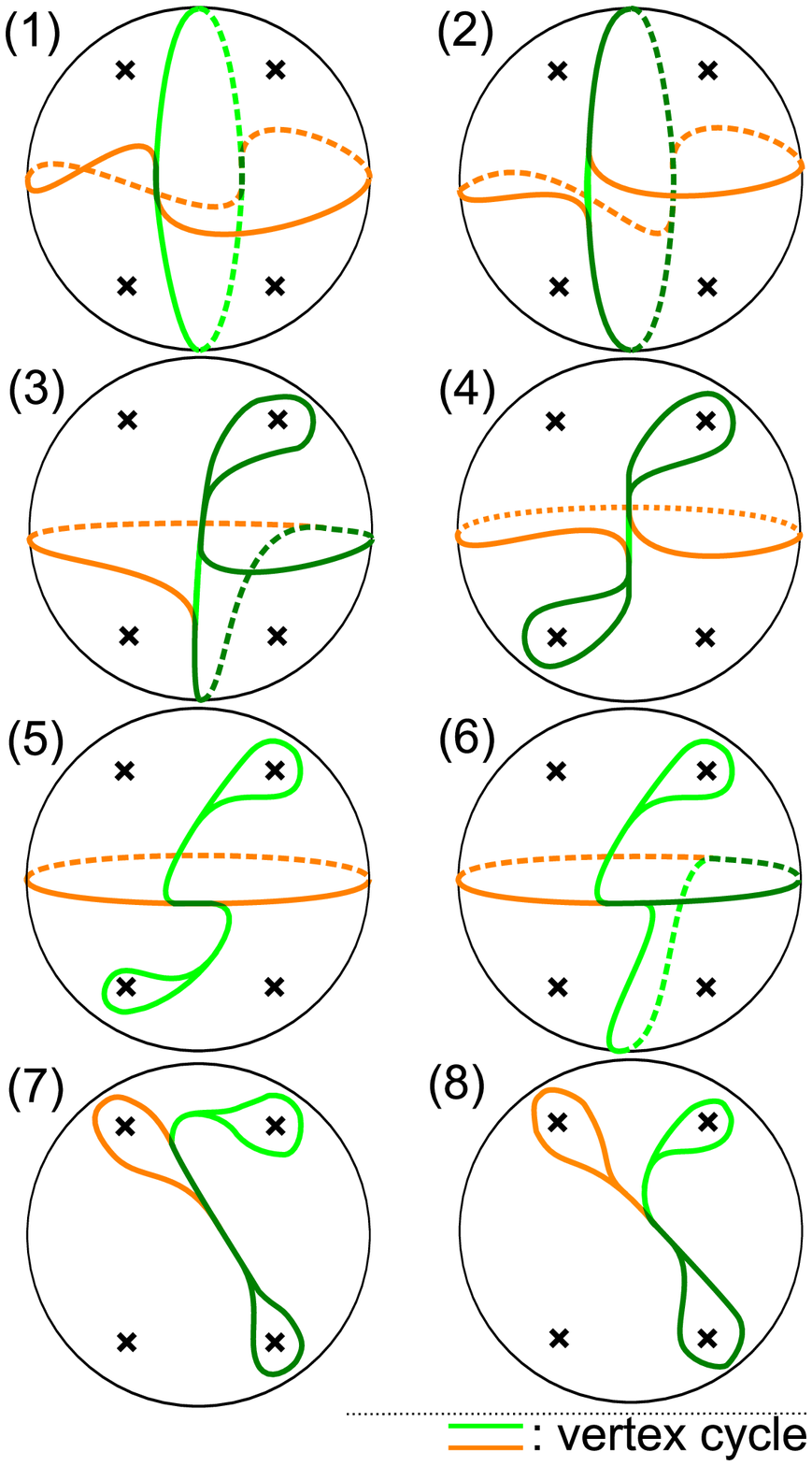}
  \end{center}
  \caption{}
  \label{fig:04-tts}
\end{figure}
We can see that all of those train tracks have exactly two vertex cycles $c, c'$
whose intersection number $i(c,c')$ equals $2$ or $4$.

First, we confirm connectivity of $\mathcal{TT}(S_{0,4})$:
\begin{prop}\label{thm:connect}
$\mathcal{TT}(S_{0,4})$ is connected.
\end{prop}

First we look at a $C^1$--diffeomorphism class which has two large edges 
and whose two vertex cycles intersect at two points ( (1) of Figure \ref{fig:04-tts} ). 
We write $A_1$ for the collection of these train tracks.

Let $\tau \in A_1$.
$\tau$ can split at two large edges $b_1$, $b_2$. 
We can get another complete train track $\tau_1$ by a right split $\tau$ at $b_1$ 
($\tau_1$ is $C^1$--diffeomorphic to (2) of Figure \ref{fig:04-tts}). 
$\tau_1$ can be right split at $b_2$ to a complete train track $\tau'$, and we can find that $\tau' \in A_1$. 
Incidentally, if we left split $\tau_1$ at $b_2$, 
we cannot get a complete train track (see Figure \ref{fig:04-splits}). 
In the same way, we can get $\tau'' \in A_1$ by being left splits $\tau$ at both $b_1$ and $b_2$. 
That is to say, we can get two complete train tracks $\tau', \tau'' \in A_1$ 
by splits at both $b_1$ and $b_2$ of $\tau$. 
\begin{figure}[htb]
  \begin{center}
    \includegraphics[keepaspectratio=true,width=110mm]{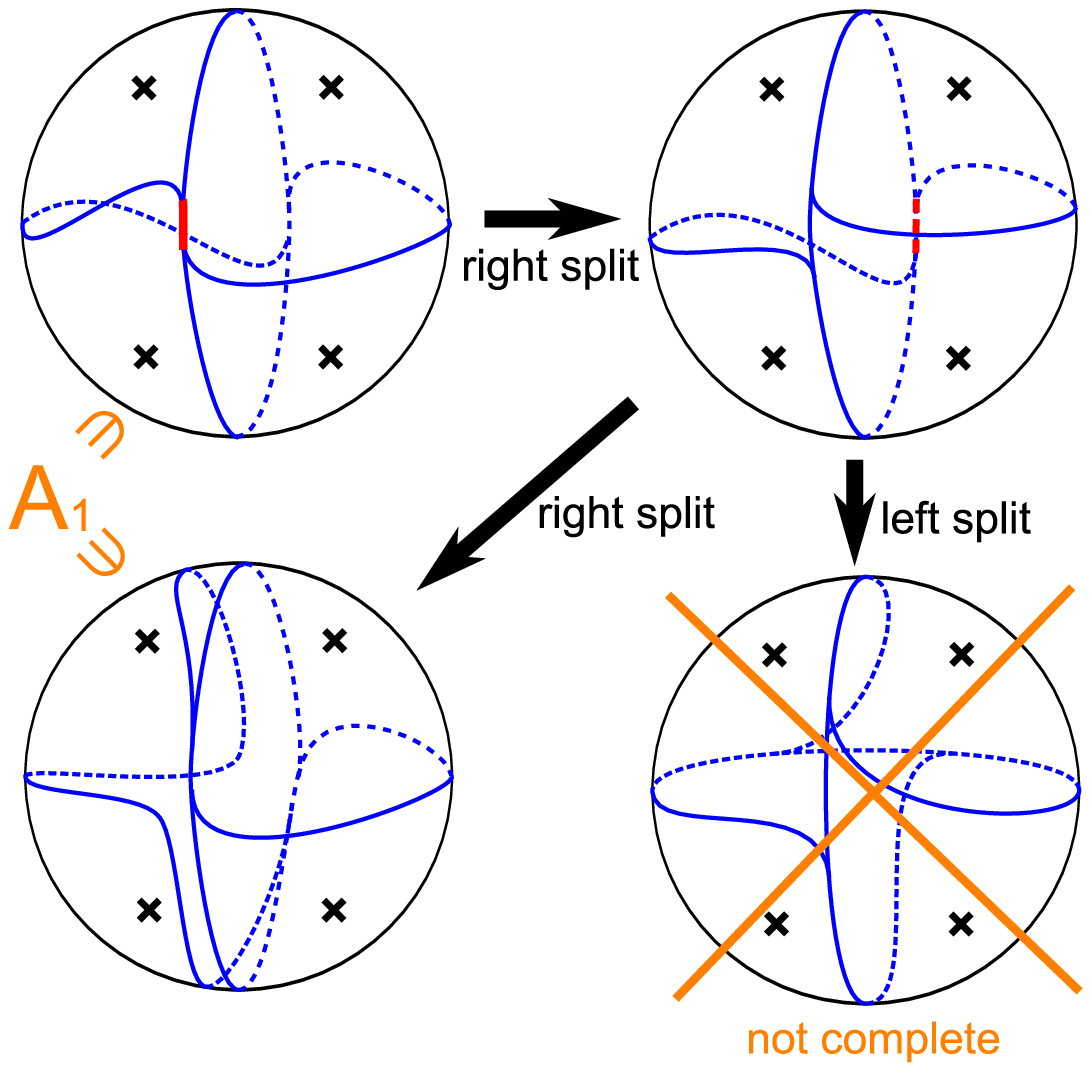}
  \end{center}
  \caption{}
  \label{fig:04-splits}
\end{figure}

We construct graph $T_1$ as follow: 
$V(T_1)$ is $A_1$. 
We connect $\tau, \tau' \in A$ by an edge 
if $\tau'$ can be obtained by splits at two larges edges of $\tau$.
Clearly, $T_1$ is homeomorphic to subgraph of $\mathcal{TT}(S_{0,4})$. 

\begin{lemma}\label{lemma:g2}
$T_1$ is connected and is quasi-isometric to the dual of $\mathcal{F}$. 
\end{lemma}

\begin{proof}
Let $\tau \in V(T_1)$. 
$\tau$ has two vertex cycles connected by an edge in $\mathcal{F}$.
Thus, we can think it just the same as $\mathcal{TT}(S_{1,1})$.
As a result, we can get this Lemma.
\end{proof}

\begin{lemma}\label{lemma:split}
Let $\tau$ be any complete train track of $S_{0,4}$. 
Then there is $\tau' \in V(T_1)$ obtained from $\tau$ by at most 5 splits. 
\end{lemma}
\begin{proof}
We can easily see that
each $C^1$--diffeomorphism class of $V(\mathcal{TT}(S_{0,4}))$ has a train track $\tau$
which implements $d(\tau, V(T_1)) \leq 5$ (see Figure \ref{fig:04-tts-split}).
Meanwhile, $d(\tau, V(T_1))$ depends only on a $C^1$--diffeomorphism class of $\tau$, 
because the mapping class group $\mathcal{M}(S_{0,4})$ acts isometrically on $\mathcal{TT}(S_{1,1})$.
Now this Lemma is proved.
\end{proof}

\begin{figure}[htbp]
  \begin{center}
    \includegraphics[scale=0.83]{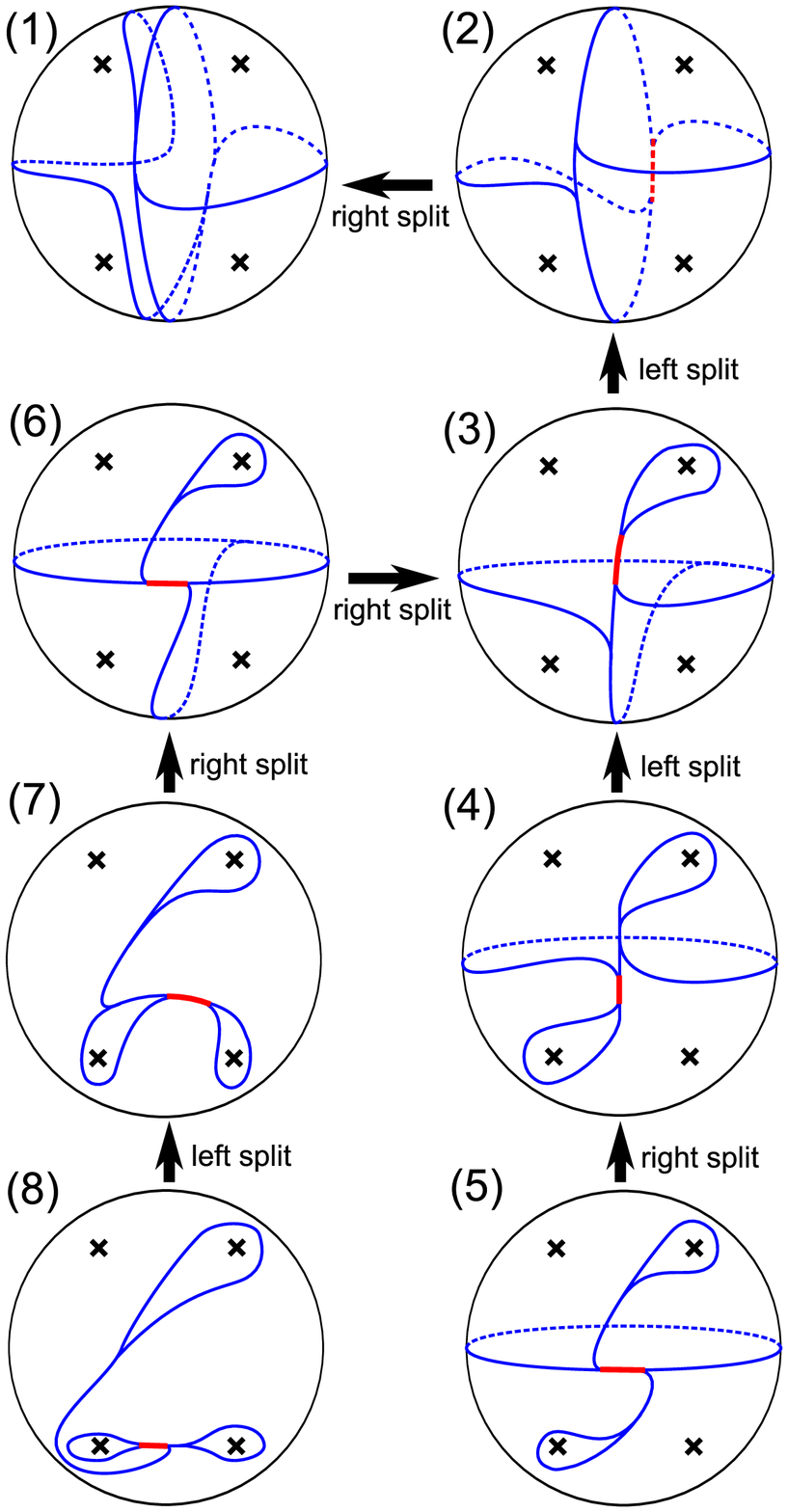}
  \end{center}
  \caption{}
  \label{fig:04-tts-split}
\end{figure}

\begin{proof}[Proof of Proposition \ref{thm:connect}]
We obtained this theorem by Lemma \ref{lemma:g2} and \ref{lemma:split}.
\end{proof}

Similar to the construction of $\varphi$ in Section \ref{sec:torus},
we construct the map $\psi : V(\mathcal{TT}(S_{0,4})) \rightarrow E(\mathcal{F})$ as follow :
Let $\tau \in V(\mathcal{TT}(S_{1,1}))$. 
Suppose $c_1, c_2$ are vertex cycles on $\tau$.
If $i(c_1, c_2) = 2$, there is an edge $e$ which connects $c_1$ and $c_2$ in $\mathcal{F}$. 
We define $\psi(\tau)$ as $e$.
If $i(c_1, c_2) = 4$, 
there are two simple closed curves $c_3, c_4$ that implement $i(c_j, c_k)=2$ $(j=1,2, \ k=3,4)$, $i(c_3, c_4)=2$. 
We define $\psi(\tau)$ as the edge which connects $c_3$ and $c_4$. 
We construct graph $G_2$ as the same as $G_1$ of Section \ref{sec:torus} and　
extends to $\psi:\mathcal{TT}(S_{0,4}) \rightarrow G_2$.　

\begin{lemma}\label{lemma:QuasiG3}
$G_2$ is quasi-isometric to the dual of the Farey graph.　
\end{lemma}
\begin{proof}
It is possible to think this just as $G_1$.　

Let $\tau \in V(\mathcal{TT}(S_{0,4}))$ and $\tau'$ obtained by a single split of $\tau$.
The relation between vertex cycles on $\tau$ and $\tau'$ can fall into the following 3 types
(see also Table \ref{table:relation}):
\begin{enumerate}
\item $\tau$ and $\tau'$ have the same vertex cycles.
Thus $\phi(\tau)$ and $\phi(\tau')$ are the same edge in $\mathcal{F}$.
\item One vertex cycle $c_1$ on $\tau$ and $c_1'$ on $\tau'$ are the same.
Another vertex cycles $c_2$ on $\tau$ and  $c_2'$ on $\tau'$ intersect at 2 points.
Thus $\phi(\tau)$ and $\phi(\tau')$ are an adjacent edge in $\mathcal{F}$. 
\item One vertex cycle $c_1$ on $\tau$ and $c_1'$ on $\tau'$ are the same. 
Another vertex cycles $c_2$ on $\tau$ and  $c_2'$ on $\tau'$ intersect at 4 points. 
Thus $\phi(\tau)$ and $\phi(\tau')$ are the next but one edge in $\mathcal{F}$. 
\end{enumerate}

So, the edges of $G_2$ connect adjacent or next but one edges in $\mathcal{F}$. 
Thus, $G_2$ is quasi-isometric to $G_1$ and the dual of $\mathcal{F}$.

\begin{table}[htb]
\begin{center}
\begin{tabular}{ccc}
\thline
diffeo. class of $\tau$ & diffeo. class of $\tau'$
	& \shortstack{relation between  \\ $\psi(\tau)$ and $\psi(\tau')$ in $\mathcal{F}$} \\ \thline\thline
(1) of Fig.\ref{fig:04-tts} & (2) & adjacent edges \\ \thline
(2) & (1) & a same edge \\ \thline
(3) & (2) & a same edge \\ \thline
(4) & (3) & adjacent edges \\ \thline
(5) & (4) & adjacent edges \\ \cline{2-3}
& (5) & next but one edges \\ \thline
(6) & (3) & adjacent edges \\ \cline{2-3}
& (6) & adjacent edges \\ \thline
(7) & (6) & a same edge \\ \thline
(8) & (7) & adjacent edges \\ \thline
\end{tabular}
\caption{relation between $\tau$ and its split $\tau'$ in $\mathcal{F}$}
\label{table:relation}
\end{center}
\end{table}

\end{proof}

We note that $G_2$ is isomorphic to Figure \ref{fig:g3}.

\begin{figure}[tb]
  \begin{center}
    \includegraphics[keepaspectratio=true,width=80mm]{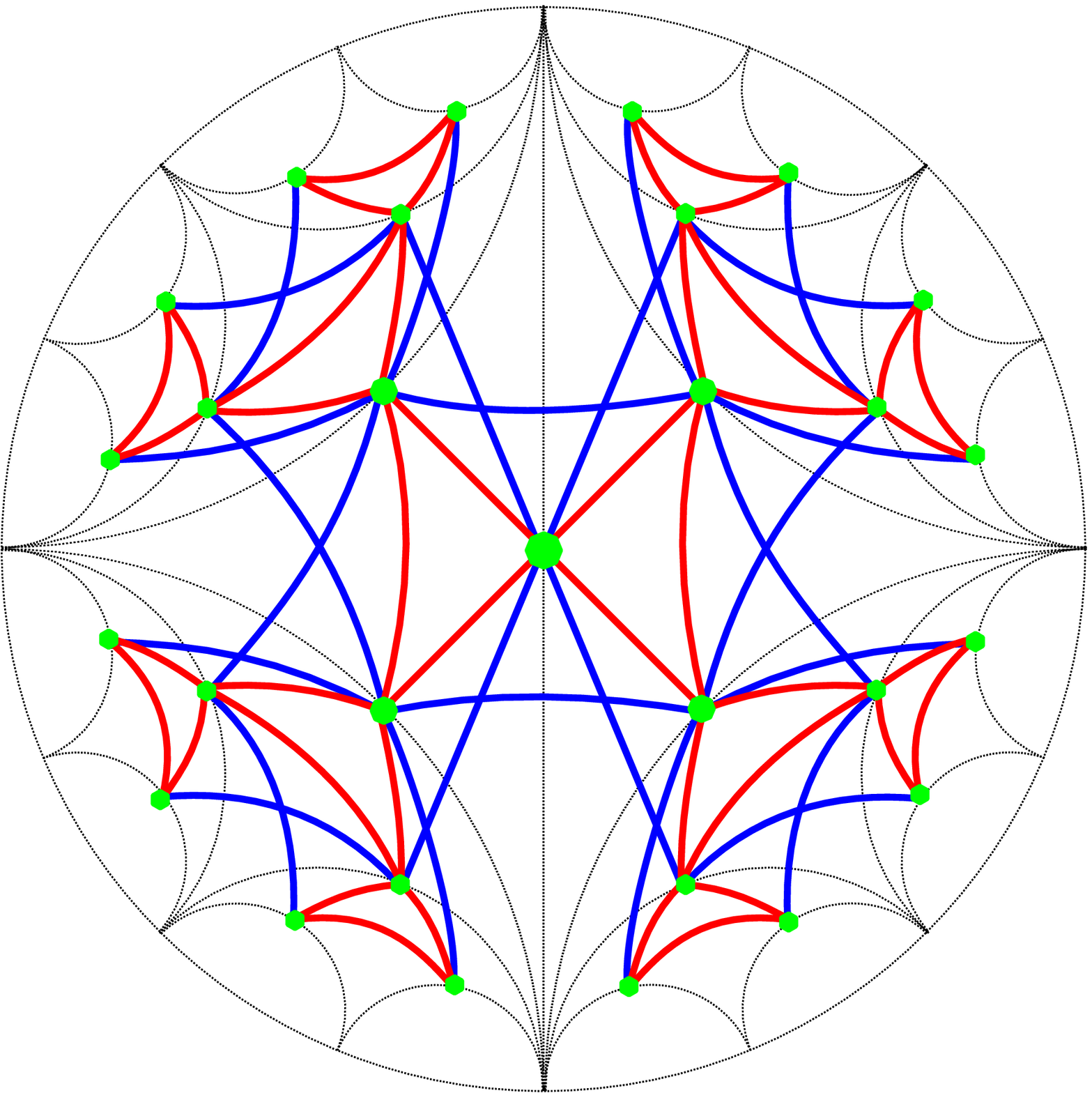}
  \end{center}
  \caption{$G_2$}
  \label{fig:g3}
\end{figure}

\begin{lemma}\label{lemma:QuasiPsi}
$\psi:\mathcal{TT}(S_{0,4}) \rightarrow G_2$ is a quasi-isometry.
\end{lemma}
\begin{proof}
It can be proved in the same way as in Lemma \ref{lemma:QuasiPhi}.

Let $\tau, \tau' \in V(\mathcal{TT}(S_{0,4}))$. 
There are only finitely many complete train tracks if vertex cycles are fixed.
Thus, $\psi^{-1}(e)$ is finite. 
Since $\mathcal{M}(S_{0,4})$ acts isometrically on $\mathcal{TT}(S_{1,1})$ 
and acts transitively on $V(G_2)$, 
$a:=diam(\psi^{-1}(e))$ is constant for all $e$.
It follows that 
$d(\psi(\tau), \psi(\tau')) \leq d(\tau, \tau') \leq (a+1)d(\psi(\tau), \psi(\tau')) + a$
and $\psi$ is a quasi-isometry.
\end{proof}

\begin{proof}[Proof of Theorem \ref{thm:main} ($4$-punctured sphere case)]
By Lemma \ref{lemma:QuasiG3} and \ref{lemma:QuasiPsi}, 
$\mathcal{TT}(S_{0,1})$ is quasi-isometric to the dual of the Farey graph.
\end{proof}

\section{Action of the mapping class group}\label{sec:act}
It is well known that $\mathcal{M}(S_{1,1})$ is isomorphic to $SL(2, \mathbb{Z})$ (see for instance \cite{Tak:01}).
Also $\mathcal{M}(S_{0,4})$ has a subgroup of finite index which is isomorphic $PSL(2, \mathbb{Z})$.

First, we prove Collolary\ref{cor:pdc} and \ref{cor:qi}

\begin{proof}[Proof of Collolary\ref{cor:pdc}]
Let $\tau$ be a complete train track on the once-punctured torus $S_{1,1}$.
The train track complex is locally finite. 
The stabilizer $stab(\tau)$ under the action of $\mathcal{M}(S_{1,1})$ is finite. 
Thus, $\{ \sigma \in \mathcal{M}(S_{1,1}) \mid d(\tau, \sigma\tau) \leq r \}$ 
is finite for all $r \geq 0$. 
It follows that the action of $\mathcal{M}(S_{1,1})$ on $\mathcal{TT}(S_{1,1})$ 
is properly discontinuous.

$\mathcal{M}(S_{1,1})$ acts transitively on $V(\mathcal{TT}(S_{1,1}))$, 
since all complete train tracks on $S_{1,1}$ are $C^1$--diffeomorphism. 
It follows that the orbit $\mathcal{M}(S_{1,1})\tau = V(\mathcal{TT}(S_{1,1}))$ and thus 
$N(\mathcal{M}(S_{1,1})\tau, 1) = \mathcal{TT}(S_{1,1})$.
That is to say, any orbit $\mathcal{M}(S_{1,1})\tau$ is cobounded. 
By Proposition \ref{ex:bow}, the action of $\mathcal{M}(S_{1,1})$ on $\mathcal{TT}(S_{1,1})$ is cocompact. 

Let $\tau$ be a complete train track on the 4-punctured sphere $S_{0,4}$. 
Just as in the case of $\mathcal{M}(S_{1,1})$, 
the action of $\mathcal{M}(S_{0,4})$ on $\mathcal{TT}(S_{0,4})$ is properly discontinuous. 

$\mathcal{M}(S_{0,4})$ acts transitively on $V(T_1)$. 
Thus the orbit $\mathcal{M}(S_{0,4})\tau$ of $\tau \in V(T_1)$ is $V(T_1)$. 
By Lemma \ref{lemma:split}, $N(\mathcal{M}(V(T_1), 6) = \mathcal{TT}(S_{0,4})$. 
So, some orbit is cobounded. 
By Proposition \ref{ex:bow}, 
the action of $\mathcal{M}(S_{0,4})$ on $\mathcal{TT}(S_{0,4})$ is cocompact. 
\end{proof}

\begin{proof}[Proof of Collolary\ref{cor:qi}]
The train track complex $\mathcal{TT}(S)$ is a locally finite graph.
Thus $\mathcal{TT}(S)$ is a proper space. 
Since the mapping class group $\mathcal{M}(S)$ is finitely generated,
by Theorem \ref{thm:bow} and Corollary \ref{cor:pdc}, 
$\mathcal{TT}(S_{1,1})$ is quasi-isometric to $\mathcal{M}(S)$. 
\end{proof}

As already stated in Section \ref{sec:torus},
$\mathcal{TT}(S_{1,1})$ is isomorphic to the Caley graph of $PSL(2,\mathbb{Z})$.
Similarly,
we can notice that $T_1$ in Section \ref{sec:4sphere} is isomorphic to $\mathcal{TT}(S_{1,1})$ 
and hence the Cayley graph of $PSL(2, \mathbb{Z})$.
Thus, $PSL(2, \mathbb{Z})$ acts freely and p.d.c. on $T_1$ and on $V(T_1)$ transitively.
Meanwhile, we can easily show that $\mathcal{M}(S_{0,4})$ acts p.d.c. on $T_1$ and
the stabilizer $stab(\tau)$ for $\tau \in T_1$ is isomorphic to 
dihedral group $\mathbb{Z}/2\mathbb{Z} \times \mathbb{Z}/2\mathbb{Z}$.
Hence index of $PSL(2, \mathbb{Z})$ on $\mathcal{M}(S_{0,4})$ equals $4$.


\end{document}